\documentclass{amsart}
\usepackage{amssymb, latexsym}

\DeclareMathOperator{\Tr}{Tr}
\DeclareMathOperator{\Alt}{Alt}

\DeclareMathOperator{\M}{\mathcal{M}}
\DeclareMathOperator{\N}{\mathcal{N}}
\DeclareMathOperator{\K}{\mathcal{K}}

\DeclareMathOperator{\rad}{rad}

\DeclareMathOperator{\Sk}{\mathcal{S}}
\DeclareMathOperator{\A}{\mathcal{A}}

\DeclareMathOperator{\Symm}{Symm}

\theoremstyle{plain}
\newtheorem{theorem}{Theorem}
\newtheorem{corollary}{Corollary}
\newtheorem{lemma}{Lemma}

\theoremstyle{definition}
\newtheorem{definition}{Definition}
\begin{document}
\title[Rank properties of subspaces]
{Rank-related dimension bounds for    \\ 
subspaces of symmetric bilinear forms}

\author[R. Gow]{Rod Gow}
\address{School of Mathematics and Statistics\\
University College Dublin\\
 Ireland}
\email{rod.gow@ucd.ie}


\begin{abstract} 

Let $V$ be a vector space of dimension $n$ over a field $K$ and 
let $\Symm(V)$ denote the space of symmetric bilinear forms defined on $V\times V$. Let $\M$ be a subspace of $\Symm(V)$.
We investigate a variety
of hypotheses concerning the rank of elements in $\M$ that lead to reasonable bounds for $\dim \M$. For example, if every non-zero element of 
$\M$ has odd rank, and $r$ is the maximum rank of the elements of $\M$, then $\dim \M\leq r(r+1)/2$ (thus $\dim \M$ is bounded independently of $n$).
This should be contrasted with the simple observation that $\Symm(V)$ contains a subspace of dimension $n-1$ in which each non-zero element
has rank $2$. 

The bound $r(r+1)/2$ is almost certainly too large, and a bound $r$ seems plausible, this being true when $K$ is finite.
We also show that the bound $\dim \M\leq r$ holds when $K$ is any field of characteristic 2. Indeed, we are able to show that a number
of results obtained for finite fields in the paper \cite{Gow2} also hold for any field of characteristic 2. 

Finally, suppose that $n=2r$, where $r$ is an odd integer, and the rank of each non-zero
element of $\M$ is either $r$ or $n$. We show that if $K$ has characteristic 2, then $\dim \M\leq 3r$. Furthermore, if $\dim \M=3r$, we obtain
interesting subspace decompositions of $\M$ and $V$ related to spreads, pseudo-arcs and pseudo-ovals.
Examples of such subspaces $\M$ exist if $K$ has an extension field
of degree $r$. We remark here that we have obtained similar theorems when $K$ is a finite field of arbitrary characteristic.

\end{abstract}
\maketitle

\section{Symmetric bilinear forms over an arbitrary field}

\noindent Let $K$ be an arbitrary field and let $V$ be a vector space of finite dimension $n$ over $K$. 
Let $\M$ be a subspace of $\Symm(V)$, where
$\Symm(V)$ 	
denotes the 
vector space of symmetric bilinear forms defined on $V\times V$.
The 
intention of this paper is to investigate if, given some hypothesis about the rank of each non-zero element of $\M$, we may
deduce an upper bound for $\dim \M$. The main hypothesis we have in mind is that the rank of each non-zero element of
$\M$ is odd. We have studied this type of situation already when $K$ is finite, but as we used counting
or divisibility arguments that depended on the finiteness of $|K|$, these methods are ineffective when $K$ is infinite. Nonetheless, we feel
that dimension bounds that apply in the finite field case probably also apply in the infinite case, and we attempt in this paper 
to extend our earlier methods so that we may consider infinite fields. Our findings are not as sharp as
those we obtained in the finite field case, but we hope to make improvements eventually.

While we are using the language and techniques of the theory of symmetric bilinear forms, it will be more convenient
briefly to employ ideas more easily expressed in terms of symmetric matrices to establish some of our foundational material. We begin
by referring to an earlier result of the author.

\begin{theorem} \label{common_zeros}
Let $\N$ be a subspace of $n\times n$ symmetric matrices over the field $K$.
Suppose that for some integer $r$ satisfying $1\leq r<n$, each non-zero element of $\N$ has rank at most $r$.
Then provided $|K|\geq r+1$, there is an invertible $n\times n$ matrix $X$, say, such that for each element $S$ of $\N$, we
have
\[
 XSX^T=\left(
\begin{array}
{cc}
        0_{n-r}&A_2\\
        A_2^T&A_1  
\end{array}
\right),
\] 
where $0_{n-r}$ is the $(n-r) \times (n-r)$ zero matrix, $A_1$ is an $r\times r$ symmetric matrix, 
and $A_2$ is an $(n-r)\times r$ matrix.

\end{theorem}

\begin{proof}
 
This follows from the proof of Theorem 1 of \cite{Gow2}. 


\end{proof}

We next note that a matrix related to $XSX^T$ above has even rank.

\begin{lemma} \label{even_rank}
 Let $r$ be an integer satisfying $1\leq r<n$. Let $A_2$ be an arbitrary $(n-r)\times r$ matrix with entries in $K$.
 Then the $n\times n$ symmetric matrix 
 \[
 \left(
\begin{array}
{cc}
        0_{n-r}&A_2\\
        A_2^T&0_r  
\end{array}
\right)
\]
 has even rank.
\end{lemma}

\begin{proof}
 This follows from the fact that a matrix and its transpose have the same rank. Alternatively, note that the product
 \[
  \left(
\begin{array}
{cc}
        0_{n-r}&A_2\\
        A_2^T&0_r  
\end{array}
\right)\left(
\begin{array}
{cc}
        -I_{n-r}&0\\
        0&I_r  
\end{array}
\right)=\left(
\begin{array}
{cc}
        0_{n-r}&A_2\\
        -A_2^T&0_r  
\end{array}
\right)
 \]
 is skew-symmetric, and hence has even rank. Since an $n\times n$ matrix has the same rank as its product with an invertible
 $n\times n$ matrix, the given symmetric matrix has even rank.

\end{proof}

We now use the results proved so far to obtain a dimension bound for an odd rank subspace of symmetric matrices.

\begin{theorem} \label{odd_dimension_bound}

Let $\N$ be a non-zero subspace of $n\times n$ symmetric matrices over $K$. Suppose that the rank of each non-zero
element of $\N$ is odd, and that $r$ is the maximum rank of the elements of $\N$. Then, provided $|K|\geq r+1$, we have
\[
 \dim \N\leq \frac{r(r+1)}{2}.
\]

\end{theorem}

\begin{proof}
 
 If $r=n$, our theorem is trivial, and so we may as well assume that $1\leq r<n$. As far as properties of rank and dimension
 are concerned, we may also assume by Theorem \ref{common_zeros} that $\N$ is contained in the subspace $\Sk$ consisting
 of all symmetric matrices of the form
 \[
 \left(
\begin{array}
{cc}
        0_{n-r}&A_2\\
        A_2^T&A_1  
\end{array}
\right).
\] 
Now $\Sk$ contains a subspace $\A$, say, consisting of all matrices of the form
\[
 \left(
\begin{array}
{cc}
        0_{n-r}&A_2\\
        A_2^T&0_r  
\end{array}
\right).
\] 
We know from Lemma \ref{even_rank} that all elements of $\A$ have even rank, and thus since all non-zero elements
of $\N$ have odd rank, by hypothesis, it follows that $\N\cap \A=0$. Finally, as $\A$ has codimension $r(r+1)/2$
in $\Sk$, we deduce that
\[
 \dim \N\leq \frac{r(r+1)}{2},
\]
as required.

\end{proof}

We next restate Theorem \ref{odd_dimension_bound} in terms of symmetric bilinear forms, and include all finite fields
in our result.

\begin{corollary} \label{odd_rank_symmetric_forms}
 Let $K$ be an arbitrary field and let $\M$ be a  non-zero subspace of $\Symm(V)$. 
 Suppose that the rank of each non-zero
element of $\M$ is odd, and that $r$ is the maximum rank of the elements of $\M$. Then we have
\[
 \dim \M\leq \frac{r(r+1)}{2}.
\]
 
\end{corollary}

\begin{proof}
 Suppose first that $K$ is finite. Then we have $\dim \M\leq r$, by Corollary 3 of \cite{Du}. Since $r\leq r(r+1)/2$,
 the desired result holds in this case. On the other hand, if $K$ is infinite, the bound follows from Theorem 
 \ref{odd_dimension_bound} proved above.
\end{proof}

We believe that the uniform dimension bound $\dim \M\leq r$ holds for all fields, not just finite fields, and we will
prove this bound for fields of characteristic 2 in the next section.

We remark here that the hypothesis concerning the oddness of the rank in Corollary \ref{odd_rank_symmetric_forms} is of course critical.
It is easy to show by construction, for example, that $\Symm(V)$ contains a subspace of dimension $n-2r+1$ in which every non-zero element has 
even rank $2r$,
whenever $2\leq 2r\leq n$.

\section{Symmetric bilinear forms in characteristic 2}

\noindent Let $K$ be a field of characteristic 2 and let $V$ be a vector space of finite dimension $n$ over $K$. 
As before, let $\Symm(V)$ denote the 
vector space of symmetric bilinear forms defined on $V\times V$, and let $\Alt(V)$ denote the subspace of alternating bilinear forms
defined on $V\times V$. We have
\[
 \dim \Symm(V)=\frac{n(n+1)}{2},\quad \dim \Alt(V)=\frac{n(n-1)}{2}
\]
and thus $\Alt(V)$ has codimension $n$ in $\Symm(V)$. It follows that a subspace of $\Symm(V)$ which intersects $\Alt(V)$ trivially has
dimension at most $n$. 

It is well known that the rank of an element of $\Alt(V)$ is even. We deduce from the remark above that if $\M$ is a subspace of $\Symm(V)$
with the property that each of its non-zero elements has odd rank, then $\dim \M\leq n$. Part of the purpose of this section of the paper is to make this
dimension bound for odd-rank subspaces of $\Symm(V)$ more precise. 

We will make good use of the following property of fields of characteristic 2.
Let $f$ be an element of $\Symm(V)$ and let $u$ and $v$ be elements of $V$. Then we have
\begin{eqnarray*}
 f(u+v,u+v)&=&f(u,u)+f(u,v)+f(v,u)+f(v,v)\\
&=&f(u,u)+2f(u,v)+f(v,v)\\
&=&f(u,u)+f(v,v).
\end{eqnarray*}
It follows that the subset of vectors $u$ in $V$ that satisfy
\[
 f(u,u)=0
\]
is a subspace of $V$. Of course, $f$ is alternating precisely when this
subspace is $V$. 

The fact that the subset of points $u$ with $f(u,u)=0$ is a subspace enables us to extend some theorems relating rank and dimension in subspaces
of $\Symm(V)$ from finite fields to arbitrary fields of characteristic 2. 

\begin{definition}
 
Let $K$ be a field of characteristic $2$ and let $\M$ be a subspace of $\Symm(V)$. We set
\[
 V(\M)=\{ v\in V: f(v,v)=0 \mbox{ for all } f\in \M\}.
\]
We also set $\M_{\Alt}=\M\cap \Alt(V)$.
\end{definition}

 It is clear from our discussion above that $V(\M)$ is a subspace of $V$. Our first useful result is a bound for $\dim V(\M)$ when $\M_{\Alt}=0$.

\begin{theorem} \label{dimension_bound_for_common_isotropic_points}
Let $\M$ be a subspace of $\Symm(V)$ with $\M_{\Alt}=0$. Then we have
\[
 \dim V(\M)\leq \dim V-\dim \M.
\]

\end{theorem}

\begin{proof}
We set $d=\dim \M$. Suppose if possible that $\dim V(\M)>n-d$. Let $U$ be a subspace of $V$ that satisfies
\[
 V=U\oplus V(\M).
\]
Our supposition on $\dim V(\M)$ implies that $\dim U\leq d-1$. 

Let $\M_U$ denote the subspace of $\Symm(U)$ obtained by restricting the elements of $\M$ to $U\times U$. We have an obvious $K$-linear
transformation from $\M$ onto $\M_U$. 

Suppose that the kernel of this linear transformation is non-zero, and let $f$ be a non-zero element of 
$\M$ whose restriction to $U\times U$ is zero. Hence, we certainly have
\[
 f(u,u)=0
\]
for all $u\in U$. We also have by definition 
\[
 f(w,w)=0
\]
for all $w\in V(\M)$. We deduce that
\[
 f(u+w,u+w)=0
\]
for all $u\in U$ and all $w\in W$. Since $U\oplus V(\M)=V$, we see that $f$ is alternating, and this contradicts our basic assumption that
$\M_{\Alt}=0$.

We have hence obtained that
\[
 \dim \M_U=\dim \M=d.
\]
Now $\M_U$ is a subspace of $\Symm(U)$ of dimension $d$ and since $\dim U\leq d-1$, it follows that
\[
 \M_U\cap \Alt(U)\neq 0.
\]
Thus there is a non-zero element $g$, say, of $\M$ whose image in $\M_U$ is alternating. We have then
\[
 g(u,u)=0
\]
for all $u\in U$ and we deduce by the previous argument that $g$ is alternating on $V\times V$. This contradiction implies that
$\dim V(\M)\leq \dim V-\dim \M$.
\end{proof}

\noindent\textbf {Remark.} Suppose that $K$ is a perfect field of characteristic 2. Then the inequality above can be replaced by the equality
\[
 \dim V(\M)=\dim V-\dim \M.
\]
We can prove this in the following way. Under the assumption that $K$ is perfect, each element $x$ of $K$ has a unique square root,
$x^{1/2}$, say. 

Consider the pairing $\M\times V\to K$ which associates to a pair $(f,v)$ the element $f(v,v)^{1/2}$. This is easily seen to be $K$-bilinear.
The left kernel is $0$, since we are assuming that $\M_{\Alt}=0$. The right kernel is $V(\M)$. Thus, by a basic result in the theory of
bilinear pairings,
\[
 \dim \M=\dim V-\dim V(\M),
\]
which is the desired equality.

We apply Theorem \ref{dimension_bound_for_common_isotropic_points} to obtain a dimension bound, in terms of rank, 
for subspaces of $\Symm(V)$ that intersect $\Alt(V)$ trivially.

\begin{theorem} \label{rank_related_dimension_bound}
Let $\M$ be a subspace of $\Symm(V)$ with $\M_{\Alt}=0$. Suppose that all elements of $\M$ have rank at most $r$. Then provided $|K|\geq r+1$,
we have $\dim \M\leq r$. 
\end{theorem}

\begin{proof}
 
We have seen that, when $\M_{\Alt}=0$, we automatically have $\dim \M\leq n$ for all fields $K$. Thus the theorem is true if $r=n$, without exception.

We may thus assume that $r<n$ and $|K|\geq r+1$. We may as well assume also that $\M$ contains some element, $f$, say, of rank $r$. Let
$R$ be the radical of $f$. It follows from Theorem 1 of \cite{Gow2} that $R$ is totally isotropic for all elements of $\M$. This implies that
\[
 R\leq V(\M). 
\]

Now $\dim R=n-r$ and Theorem \ref{dimension_bound_for_common_isotropic_points} tells us that
\[
\dim V(\M)\leq n-\dim \M.
\]
Since we know now that $\dim R\leq \dim V(\M)$, we deduce that
\[
 n-r\leq n-\dim \M
\]
and hence $\dim \M\leq r$, as required.
\end{proof}

\begin{corollary} \label{common_radicals}
 
Let $\M$ be an $r$-dimensional constant rank $r$ subspace of $\Symm(V)$. Suppose also that $|K|\geq r+1$. Then all non-zero
elements of $\M$ have the same radical, which equals $V(\M)$, provided that $\M_{\Alt}=0$. In particular, this is the case if $r$ is odd.
\end{corollary}

\begin{proof}
 
Let $f$ be any non-zero element of $\M$ and let $R$ be the radical of $f$. We showed in the proof of 
Theorem \ref{rank_related_dimension_bound} that $R\leq V(\M)$, since we are assuming that $|K|\geq r+1$.

Now suppose that $\M_{\Alt}=0$. It follows that
\[
 \dim V(\M)\leq \dim V-\dim \M=n-r
\]
by Theorem \ref{dimension_bound_for_common_isotropic_points}. Since $\dim R=n-r\leq \dim V(\M)$, it follows that $\dim V(\M)=n-r$ and $R=V(\M)$.
Thus all non-zero elements of $\M$ have the same radical, which is the subspace $V(\M)$.

Finally, suppose that $r$ is odd. Then we automatically have $\M_{\Alt}=0$, since the rank of any element of $\Alt(V)$ is even, and thus all radicals
equal $V(\M)$ in this case. 
\end{proof}

Our dimension bound for odd constant rank subspaces of $\Symm(V)$ enables us to obtain another dimension bound for subspaces of
$\Symm(V)$ in which exactly two non-zero ranks occur, one being $n$, the other $r$ say, where $0<r<n$, $r$ is odd and $n$ even.

\begin{theorem} \label{two_rank_subspace}
 
Let $\M$ be a subspace of $\Symm(V)$ and suppose that the rank of each non-zero element of $\M$ is either $r$ or $n$, where $0<r<n$, $r$ is odd 
and $n$ even.  Then we have
$\dim \M\leq n+r$. 
\end{theorem}

\begin{proof}
 Let $V^*$ denote the dual space of $V$ and let $u$ be any non-zero element of $V$. We define a linear transformation $\epsilon_u:\M\to V^*$
by setting
\[
 \epsilon_u(f)(v)=f(u,v)
\]
for all $f\in\M$ and all $v\in V$. Let $\K_u=\ker \epsilon_u$. Then, since $\dim V^*=n$, we have
\[
 \dim \M\leq \dim \K_u+n.
\]
We note that $\K_u$ is the subspace of $\M$ consisting of those forms that contain $u$ in their radical. Thus, each element
of $\K_u$ has rank $r$ and we see that $\K_u$ is a constant rank $r$ subspace of $\Symm(V)$. 

Now if $K$ is finite, we have $\dim \K_u\leq r$ by Theorem 2 and Corollary 3 of \cite{Du}, whereas if $K$ is infinite, we have the same dimension
bound by Theorem \ref{rank_related_dimension_bound}. It follows that $\dim \M\leq n+r$.

\end{proof}

The bound $\dim \M\leq n+r$ obtained above is probably not optimal except when $n$ is even and $r=n/2$ is odd. We suspect that the bound
$\dim \M\leq 2n-r$ holds in general, and this is a better bound. 

 Suppose that $n=2r$, where $r$ is odd. 
The subspaces $\K_u$ described above are $r$-dimensional constant rank $r$ subspaces in $\Symm(V)$ and this enables us to obtain 
an interesting decomposition theorem for $\M$ and $V$, described below,
if we assume that $|K|\geq r+1$ and $\dim \M=3r$.

\begin{theorem} \label{spread_of_V}
 Suppose that $n=2r$, where $r$ is an odd positive integer.
Let $\M$ be subspace of $\Symm(V)$ of dimension $3r$ in which
the rank of any non-zero element is $r$ or $n$. Suppose also that $K$ has characteristic $2$ and $|K|\geq r+1$. 

Let $\N$ be any $r$-dimensional constant rank $r$ subspace of $\M$. Then $\N=\K_u$, for some $u\in V$. 

Furthermore, let $\N_1$ be any other  $r$-dimensional constant rank $r$ subspace of $\M$ different from $\N$. Then we have $\N\cap \N_1=0$.

The common radicals of the non-zero elements of the $r$-dimensional constant rank $r$ subspaces of $\M$ 
form a spread of $r$-dimensional subspaces that cover $V$. 

The $n$-dimensional subspace $\N\oplus \N_1$ contains no non-zero alternating elements and no elements of rank $r$ other than the elements in $\N$ 
and $\N_1$. Thus 
\[
(\N\oplus \N_1)\cap \N_2=0,\quad \M=\N\oplus \N_1\oplus \N_2
\] 
for any $r$-dimensional constant rank $r$ subspace $\N_2$ different from $\N$ and $\N_1$.

We also have
\[
 \M=\N\oplus \N_1\oplus \M_{\Alt}.
\]
$\M_{\Alt}$ is thus an $r$-dimensional subspace of alternating bilinear forms of rank $n$.

\end{theorem}

\begin{proof}
 Let $\N$ be any $r$-dimensional constant rank $r$ subspace in $\M$.  Since we are assuming that $|K|\geq r+1$,   
Corollary \ref{common_radicals} shows that all the non-zero elements in $\N$ have the same radical, $U$, say, where $U$ is an $r$-dimensional subspace of 
$V$. 

Suppose now that $\N_1$ is a different $r$-dimensional constant rank $r$ subspace of $\M$. We claim that $\N\cap \N_1=0$. For suppose that
$f$ is a non-zero element in $\N\cap \N_1$. Then since all the elements of $\N_1$ have the same radical, again by Corollary
\ref{common_radicals}, 
and $f$ is an element of $\N_1$ with radical $U$, it follows that all elements of $\N_1$ also have radical equal to $U$. 
Now if we 
consider the subspace $\N+\N_1$ of $\M$, we see that $U$ is contained in the radical of each element of $\N+\N_1$. This implies, by the rank 
properties of the elements of $\M$, that all non-zero elements of $\N+\N_1$ have rank $r$. Thus, $\N+\N_1$ is a constant rank $r$ subspace
of dimension greater than $r$, and this contradicts Theorem \ref{rank_related_dimension_bound}. We deduce that $\N\cap \N_1=0$, as claimed.

Let $U_1$ be the common radical of the non-zero elements of $\N_1$. We claim that $U\cap U_1=0$. 
For suppose that $u$ is a non-zero element of $U\cap U_1$. Such an element $u$ is in the radical of all elements of $\N$ and of all elements of $\N_1$.
But then $u$ is in the radical of all elements of $\N+\N_1$. This implies that all non-zero elements of $\N+\N_1$ have rank
$r$, which is impossible by the argument above. Thus, $U\cap U_1=0$, as claimed. Furthermore, since any non-zero element $w$ of $V$ is in the
common radical of the non-zero elements of the $r$-dimensional constant rank $r$ subspace $\K_w$, we see that the common radicals form a spread
of subspaces of dimension $r$ that covers $V$.

Since $U$ and $U_1$ both have dimension $r$ and intersect trivially, we have $V=U\oplus U_1$. Let now $g$ and $g_1$ be non-zero elements of $\N$
and $\N_1$, respectively. We claim that $g+g_1$ has rank $n$, and is not alternating. For, since $g_1$ is identically zero on $U_1\times U_1$,
and has rank $r$ on $V\times V$, it has rank $r$ on $U\times U$. Likewise, $g$ has rank $r$ on $U_1 \times U_1$ and is identically zero on
$U\times U$. Thus since $V=U\oplus U_1$ and this decomposition is orthogonal with respect to $g$ and $g_1$, we deduce that $g+g_1$ has rank $2r=n$.
This implies that
the only elements of rank $r$ in $\N\oplus \N_1$ are those in $\N$ or in $\N_1$.

Furthermore, suppose if possible that $g+g_1$ is alternating on $V\times V$. It would follow that $g$ is alternating on $U_1\times U_1$, which is 
impossible as $g$ has odd rank $r$ on $U_1\times U_1$. Thus $g+g_1$ is also non-alternating, as claimed. 

Let $\N_2$ be any $r$-dimensional constant rank $r$ subspace different from $\N$ and $\N_1$. We have just shown that the only
elements of rank $r$ in $\N\oplus \N_1$ are those in $\N$ or in $\N_1$. Since we have already proved that $\N\cap N_2=\N_1\cap \N_2=0$,
it follows that $(\N\oplus \N_1)\cap \N_2=0$. Dimension arguments readily imply that $\M=\N\oplus \N_1\oplus \N_2$.

Finally, we know that $\M_{\Alt}$ has codimension at most $n$ in $\M$, and we have shown that it intersects $\N\oplus \N_1$ trivially.
Dimension arguments again show that $ \M=\N\oplus \N_1\oplus \M_{\Alt}$.

\end{proof}

\noindent\textbf{Some Constructions and Comments.}
It is reasonable to enquire whether examples of fields $K$ exist where the hypotheses of Theorem \ref{spread_of_V} are satisfied and the upper
bound for $\dim \M$ is met. Here we describe such examples.

Suppose that $K$ has an extension field $L$, say, of odd degree $r>1$ over $K$. Let $W$ be a two-dimensional vector space over $L$ and let
$V$ denote $W$ considered as a vector space of dimension $2r$ over $K$. Let $\Tr$ denote the trace form from $L$ into $K$. We note that
$\Tr$ is not identically zero, since it is the identity on $K$.

Given $f\in \Symm(W)$, we define $f'\in \Symm(V)$ by setting
\[
 f'(u,v)=\Tr(f(u,v))
\]
for all $u$ and $v$ in $V$. Let $\M$ denote the subspace of $\Symm(V)$ consisting of all such forms $f'$. Then it is easy to see that
$\dim \M=3r$, and the non-zero forms in $\M$ have rank $r$ or $n=2r$. Thus $\M$ meets our requirements.

We might also ask if the hypothesis in Theorem \ref{spread_of_V} that $r$ is odd is really necessary. The following simple examples
show that our methods and deductions are sensitive to the parity of $r$.

Let $V$ be a 4-dimensional vector space over $K$, where $K$ is again a field of characteristic 2. We set $\M=\Alt(V)$, which is a 6-dimensional
subspace of $\Symm(V)$ in which the non-zero elements have rank 2 or 4. It is straightforward to see that $M$ contains 3-dimensional
constant rank 2 subspaces.

On the other hand, suppose that $K$ has a separable extension field $L$ of degree 2 over $K$. Then we may construct a 6-dimensional
subspace $\M_1$, say, of $\Symm(V)$ from $\Symm(W)$, where $W$ is a 2-dimensional space of $L$ (as we did above using an odd degree extension
of $K$). The non-zero elements of $\M_1$ also have rank 2 or 4, but there are no 3-dimensional constant rank 2 subspaces in $\M_1$.

In the case that $K=\mathbb{F}_q$, where $q$ is a power of 2, we calculate that $\M$ contains exactly $(q^2+1)(q^3-1)$ elements of rank 2,
whereas $\M_1$ contains exactly $q^4-1$ elements of rank 2, which is a considerably smaller number.

If we perform these constructions over $\mathbb{F}_{q^r}$, where $r$ is an arbitrary positive integer, rather than $\mathbb{F}_q$, 
we obtain 6-dimensional subspaces $\N$, $\N_1$, say, of symmetric bilinear forms defined on a 4-dimensional space over
$\mathbb{F}_{q^r}$. The rank of each non-zero element of
$\N$ and of $\N_1$ is 2 or 4. The process of restriction of scalars to $\mathbb{F}_{q}$ yields two $6r$-dimensional subspaces of symmetric bilinear
forms defined on a $4r$-dimensional space over $\mathbb{F}_{q}$ in which the rank of each non-zero element is $2r$ or $4r$.
The two subspaces clearly contain different numbers of elements of rank $2r$, the number being $(q^{2r}+1)(q^{3r}-1)$ in one case,
$q^{4r}-1$ in the other. Thus there is no analogue of Theorem \ref{spread_of_V} in the case of finite fields of characteristic 2 
when we allow $r$ to be even
(we can certainly say that at least two different structures exist).

\section{Improved versions of Corollary 2 for finite fields}

\noindent Using the fact that the radicals of non-zero elements in constant rank subspaces of $\Symm(V)$ are subspaces of $V(\M)$ if the underlying
field $K$ of characteristic 2 is sufficiently big, we will show in this section how  Corollary \ref{common_radicals} concerning the equality of radicals can be improved
for finite fields. We require some preliminary work bounding the number of subspaces of a given dimension in a vector space over a finite field.

\begin{lemma} \label{elementary_inequality}
 
 Let $a$ and $b$ be integers satisfying $1\leq b\leq a$, and let $x\geq 2$ be a real number. Then
 \[
  \frac{x^a-1}{x^b-1}<x^{a-b}(1+x^{1-b}).
 \]

\end{lemma}

\begin{proof}
 
 We have
 \[
  (x^b-1)x^{a-b}(1+x^{1-b})-(x^a-1)=x^{a-b}(x+x^{b-a}-1-x^{1-b}).
 \]
 Now, since we are assuming that $x\geq 2$ and $b\geq 1$, we have
 \[
  x\geq 1+x^{1-b},
 \]
and the inequality follows, since $x^{b-a}>0$ also.

\end{proof}

\begin{lemma} \label{advanced_inequality}
 
 Let $a$ and $b$ be integers satisfying $1\leq b\leq a$, and let $x\geq 4$ be a real number. Then if we define
 \[
  F(x)=\frac{(x^a-1)(x^{a-1}-1)\cdots (x^{a-b+1}-1)}{(x-1)(x^2-1)\cdots (x^b-1)},
 \]
we have $F(x)<4x^{(a-b)b}$.
 
\end{lemma}

\begin{proof}
 Consider a term
 \[
  \frac{x^{a-i+1}-1}{x^i-1}
 \]
 in the product representation of $F(x)$. We have
 \[
  \frac{x^{a-i+1}-1}{x^i-1}<x^{a-2i+1}(1+x^{1-i}),
 \]
by Lemma \ref{elementary_inequality}. Thus,
\[
 F(x)<x^A\prod_{i=1}^b (1+x^{1-i}),
\]
where
\[
 A=\sum_{i=1}^b (a-2i+1)=(a-b)b.
\]
To prove the inequality for $F(x)$, it therefore suffices to show that
\[
 \prod_{i=1}^b (1+x^{1-i})<4
\]
when $x\geq 4$. 

Now if $z$ is a positive real number, it is well known that
\[
1+z< e^z.
\]
It follows that
\[
 \prod_{i=1}^b (1+x^{1-i})<e^S,
\]
where
\[
 S=\sum _{i=1}^b x^{1-i}.
\]
Since we are assuming that $x\geq 4$,
\[
 S<\sum_{i=1}^\infty 4^{1-i}=\frac{4}{3}.
\]
Finally, $e^{4/3}<4$, since $e^4<55<64=4^3$. This yields the bound $F(x)<4x^{(a-b)b}$, as required.

\end{proof}

\begin{theorem} \label{common_radicals_small_dimension}

Let $q$ be a power of $2$ and let $V$ be a vector space of dimension $n$ over $\mathbb{F}_q$. Let $r$ be an odd  integer
satisfying $1<r<n$. Let $\M$ be a $d$-dimensional constant rank $r$ subspace of $\Symm(V)$. Then if $q\geq r+1$ and
\[
 r\geq d> \frac{2(n-r)r}{2(n-r)+1},
\]
all the non-zero elements of $\M$ have the same radical.
 
\end{theorem}

\begin{proof}
 The radical of any non-zero element of $\M$ has dimension $n-r$, and since we are assuming that $q\geq r+1$, the radical is contained
 in the subspace $V(\M)$, by the proof of Corollary \ref{common_radicals}. Furthermore, $\dim V(\M)=n-d$, by the remark following
 the proof of Theorem \ref{dimension_bound_for_common_isotropic_points}, since we are working over a finite, and hence perfect, field.
 
 Suppose that
 not all the non-zero elements of $\M$ have the same radical. Let $R_1$, \dots, $R_t$ be the different
$(n-r)$-dimensional subspaces of $V$ that occur as the radicals of the non-zero elements of $\M$, where $t>1$. 

We set
\[
 \M_i=\{ f\in \M: R_i \leq \rad\,f\},\, 1\leq i\leq t.
\]
It is easy to see that $\M_i\cap \M_j=0$ if $i\neq j$, and $\M$ is the union of the subspaces $\M_i$. Thus we have a partition
of $\M$ into $t>1$ non-trivial subspaces. It follows from Lemma 1 of \cite{Gow2} that
\[
 t\geq q^{d/2}+1
\]
if $d$ is even, and
\[
 t\geq q^{(d+1)/2}+1
\]
if $d$ is odd. 

The $R_i$ are $t$ different $(n-r)$-dimensional subspaces of the $(n-d)$-dimensional subspace $V(\M)$. It is well known that the total number
of $(n-r)$-dimensional subspaces of an $(n-d)$-dimensional space is the $q$-binomial coefficient
\[
 \frac{(q^{n-d}-1)\cdots (q^{r-d+1}-1)}{(q-1)\cdots (q^{n-r}-1)}.
\]
Lemma \ref{advanced_inequality} implies that
\[
 t\leq 4q^{(n-r)(r-d)},
\]
since we are assuming that $q\geq r+1\geq 4$. 

Suppose first that $d$ is even. We then have
\[
 q^{d/2}+1<4q^{(n-r)(r-d)}.
\]
We claim that in this case
\[
 \frac{d}{2}\leq (n-r)(r-d).
\]
For if this inequality does not hold, we must rather have
\[
 \frac{d}{2}\geq (n-r)(r-d)+1
\]
and hence 
\[
 q^{d/2}+1>q^{d/2}\geq qq^{(n-r)(r-d)}\geq 4q^{(n-r)(r-d)},
\]
since we are assuming that $q\geq 4$. This is a contradiction to our previous inequality and therefore we indeed have
\[
 \frac{d}{2}\leq (n-r)(r-d).
\]

When we solve the inequality for $d$, we obtain
\[
 d\leq \frac{2(n-r)r}{2(n-r)+1}.
\]
Thus, if we take
\[
 d>\frac{2(n-r)r}{2(n-r)+1}
\]
(and recall that $d$ is an integer), it must be the case that all the radicals are the same.

Similarly, if $d$ is odd, and $t>1$, we must have
\[
  \frac{d+1}{2}\leq (n-r)(r-d).
\]
This leads to the inequality
\[
 d\leq \frac{2(n-r)r}{2(n-r)+1}-\frac{1}{2(n-r)+1}.
\]
Thus, if we again take
\[
 d>\frac{2(n-r)r}{2(n-r)+1},
\]
the radicals are certainly all the same.

\end{proof}

We can think of Theorem \ref{common_radicals_small_dimension} in the following way. We know from Corollary \ref{common_radicals}
that all the non-trivial radicals are the same in an $r$-dimensional constant rank $r$ subspace of $\Symm(V)$ when $q\geq r+1$. We can
diminish the dimension by the fraction $2(n-r)/(2(n-r)+1)<1$ and still obtain an identical conclusion about the radicals. The theorem
has no content, beyond being a restatement of Corollary \ref{common_radicals}, when $r\leq (2n-1)/3$.

\bigskip

\noindent\textbf{Acknowledgement.} We are grateful to John Sheekey for conversations relating to Theorem 6. The subject matter of the
theorem occurs in the topic of pseudo-arcs and pseudo-ovals in geometry.

\end{document}